\documentclass[a4paper,11pt]{amsart}

\usepackage[utf8]{inputenc}
\usepackage{lmodern}
\usepackage{url}
\usepackage{amsmath, amsthm, amssymb, amscd, accents, bm}
\usepackage{mathtools}
\usepackage{mdwlist}
\usepackage{paralist}
\usepackage{graphicx}
\usepackage{epstopdf}
\usepackage{datetime}
\usepackage{afterpage}
\usepackage[dvipsnames]{xcolor}
\usepackage{hyperref}
\usepackage{soul}
\usepackage{colonequals}

\usepackage[sort,nocompress]{cite}
\mathtoolsset{showonlyrefs}

\newtheorem{definition}{Definition}[section]
\newtheorem{theorem}[definition]{Theorem}
\newtheorem{lemma}[definition]{Lemma}
\newtheorem{corollary}[definition]{Corollary}

\newtheorem{remark}[definition]{Remark}

\def\N{{\mathbb N}}
\def\Z{{\mathbb Z}}
\def\R{{\mathbb R}}

\def\C{{\mathbb C}}
\def\Q{{\mathbb Q}}

\newcommand{\lt}{{L^2(\R)}}

\newcommand{\re}{{\mathrm{Re} \,}}

\newcommand{\ft}{{\mathcal{F}}}

\begin{document}

\title[Arithmetic progressions and holomorphic phase retrieval]{Arithmetic progressions and holomorphic phase retrieval}

\author[Lukas Liehr]{Lukas Liehr}

\address{Faculty of Mathematics, University of Vienna, Oskar-Morgenstern-Platz 1, 1090 Vienna, Austria}
\email{lukas.liehr@univie.ac.at}


\begin{abstract}
We study the determination of a holomorphic function from its absolute value. Given a parameter $\theta \in \mathbb{R}$, we derive the following characterization of uniqueness in terms of rigidity of a set $\Lambda \subseteq \mathbb{R}$: if $\mathcal{F}$ is a vector space of entire functions containing all exponentials $e^{\xi z}, \, \xi \in \mathbb{C} \setminus \{ 0 \}$, then every $F \in \mathcal{F}$ is uniquely determined up to a unimodular phase factor by $\{ |F(z)| : z \in e^{i\theta}(\R + i \Lambda)  \}$ if and only if $\Lambda$ is not contained in an arithmetic progression $a\mathbb{Z}+b$. Leveraging this insight, we establish a series of consequences for Gabor phase retrieval and Pauli-type uniqueness problems. For instance, $\Z \times \tilde \Z$ is a uniqueness set for the Gabor phase retrieval problem in $L^2(\R_+)$, provided that $\tilde \Z$ is a suitable perturbation of the integers.
\end{abstract}

\maketitle

\section{Introduction and results}\label{sec:1}

\subsection{General statement}

Let $\mathcal F$ be a set of holomorphic functions defined on some domain $\Omega$. The uniqueness problem in \emph{holomorphic phase retrieval} concerns the question whether functions in $\mathcal{F}$ are distinguishable from their absolute values on some subset $\mathcal{U} \subseteq \Omega$, that is, from the values $\{ |F(u)| : u \in \mathcal{U} \}$.
It is clear that in the absence of phase information, every function $F$ is indistinguishable from $\tau F$, provided that $\tau \in \C$ satisfies $|\tau| = 1$. Hence, the ambiguity of a multiplicative factor with unit modulus is unavoidable. Motivated by this, we define for two complex-valued functions $F,H$ (not necessarily holomorphic) an equivalence relation via
$$
F \sim H \iff  \left( \, \exists \tau \in \C, \ |\tau|=1 : \ F=\tau H \, \right).
$$
We say that $\mathcal{U} \subseteq \Omega$ \emph{is a uniqueness set for the phase retrieval problem in} $\ft$, if for every $F,H \in \ft$, we have
$$
|F(z)| = |H(z)|, \ z \in \mathcal{U} \implies F \sim H.
$$
The holomorphic phase retrieval problem arises naturally from phase retrieval problems associated with specific integral transforms which are intimately linked to spaces of holomorphic functions. This includes, for instance, the Fock space induced by the Gabor transform \cite{GrohsRathmair,grohsliehr1}, Paley-Wiener spaces induced by the Fourier transform of compactly supported functions \cite{Lai2021,bianchi}, or Bergman and Hardy spaces which are related to the Cauchy Wavelet transform \cite{Mallat2015,bart}. The general theory on holomorphic phase retrieval is impacted by these special cases.

It was recently established in a series of papers, that sets $\mathcal{U}$ which are equidistant parallel lines or lattices, do not form uniqueness sets for the phase retrieval problem induced by the Gabor transform \cite{grohsLiehr3,alaifari2020phase,grohsLiehr4}. The present article centers around the following question.

\vspace{0.2cm}

\begin{quote}
{\it 
    Is this non-uniqueness property a consequence of a general principle associated with the rigid structure of the considered sets $\mathcal{U}$, and the underlying analyticity?
    }
\end{quote}

\vspace{0.2cm}

Indeed, the following characterization in the general setting of entire functions shows that uniqueness fails precisely in the rigid setting. This constitutes the main result of the article. Its consequences related to various other phase retrieval problems are stated in the upcoming sections. In order to formulate our results in a precise manner, we introduce the following notation: for a subset $\Xi \subseteq \C$, we denote by $\mathcal{E}(\Xi) \coloneqq \{ z \mapsto e^{\xi z} : \xi \in \Xi \}$ the system of all exponentials with complex frequencies in $\Xi$. Further, we set $\C^* \coloneqq \C \setminus \{ 0 \}$. Recall, that an arithmetic progression is a set of the form $a\Z+b$ for some $a,b \in \R$ (we do not exclude the choice $a=0$). Finally, we note that all vector spaces are understood to be over the complex field.

\begin{theorem}\label{thm:main}
Let $\mathcal{F}$ be a vector space of entire functions satisfying $\mathcal{E}(\C^*) \subseteq \ft$. Further, let $\theta \in \R$, and let $\Lambda \subseteq \R$. Then $e^{i\theta}(\R + i \Lambda)$ is a uniqueness set for the phase retrieval problem in $\ft$ if and only if $\Lambda$ is not contained in an arithmetic progression.
\end{theorem}

For instance, if $a \neq 0$, and if $b \in \R$, then every set $\Lambda$ of the form
\begin{equation}\label{eq:perturbation}
    \Lambda = \{ an+b+p_n : n \in \Z \}, \quad p_n \neq 0, \quad \lim_{|n| \to \infty} p_n = 0,
\end{equation}
is not contained in an arithmetic progression. This set constitutes a perturbation of the arithmetic progression $a\Z+b$ with perturbation sequence $(p_n)_{n \in \Z}$. Theorem \ref{thm:main} thus shows, that if $\Lambda$ is an arbitrary perturbation of the form \eqref{eq:perturbation}, then $e^{i\theta}(\R + i \Lambda)$ is a uniqueness set for the holomorphic phase retrieval problem.

\subsection{Gabor phase retrieval from parallel lines}

The short-time Fourier transform (STFT) of $f \in \lt$ with respect to the window function $g \in \lt$ is defined by
$$
V_gf(x,\omega) = \int_\R f(t)\overline{g(t-x)} e^{-2\pi i \omega t} \, dt.
$$
For $\varphi(t)=e^{-\pi t^2}$ we define $\mathcal{G} \coloneqq V_\varphi$, and call $\mathcal{G}$ the Gabor transform. We say that $\mathcal{U} \subseteq \R^2$ \emph{is a uniqueness set for the Gabor phase retrieval problem in} $X \subseteq \lt$, if for every $f,h \in X$ we have
$$
|\mathcal{G}f(z)| = |\mathcal{G}h(z)|, \ z \in \mathcal{U} \implies f \sim h.
$$
Lately, there has been a surge of research dedicated to determining uniqueness sets $\mathcal U$ for the Gabor phase retrieval problem in different function spaces $X$ \cite{multiwindow, grohsliehr1}, and related regimes, such as finite-dimensional spaces \cite{Bojarovska2016}, and group theoretical settings \cite{fuhr2023phase,Bartusel2023}. The problem can be viewed as a special instance of a frame-theoretical formulation of phase retrieval, a topic extensively studied in references such as \cite{alharbi1, alharbi2, daub, bandeira_savingphase, Bodmann2015,CONCA2015346,alaifariGrohs,freeman2023stable}.

Alaifari and Wellershoff showed in \cite{alaifari2020phase}, that a set of the form $\mathcal{U} = R_\theta(\R \times (a\Z+b))$ cannot form a uniqueness set for the Gabor phase retrieval problem in $\lt$. Here, $R_\theta : \R^2 \to \R^2, \, (x,y) \mapsto (x\cos \theta - y \sin \theta, \, x \sin \theta + y \cos \theta)$, represents a rotation by an angle $\theta$. This result has been extended upon by Grohs and the author, encompassing arbitrary window functions \cite{grohsLiehr3}, and arbitrary dimensions \cite{grohsLiehr4}. The intimate relation between the Gabor transform of a function and elements in the Fock space, when combined with Theorem \ref{thm:main}, implies that the impossibility to achieve uniqueness from sets of the form $R_\theta(\R \times \Lambda)$ occurs precisely when $\Lambda$ takes the structure $a\Z +b$.

\begin{theorem}\label{thm:gabor_impl}
    Let $\theta \in \R$, and let $\Lambda \subseteq \R$. Then $R_\theta(\R \times \Lambda)$ is a uniqueness set for the Gabor phase retrieval problem in $\lt$ if and only if $\Lambda$ is not contained in an arithmetic progression.
\end{theorem}

\subsection{Intermediate spaces}

The findings in \cite{alaifari2020phase,grohsLiehr3,grohsLiehr4} demonstrate that achieving uniqueness for the Gabor phase retrieval problem in $\lt$ is not possible when using samples located on a lattice. A lattice refers to a set $\Gamma \subseteq \R^2$ that can be expressed as $\Gamma=A\Z^2$, where $A \in \mathrm{GL}_2(\R)$. However, it has been proven that lattice-uniqueness holds true in the spaces $L^2[0,c]$, where $0<c<\infty$, that is, functions in $\lt$ that vanish almost everywhere outside the compact interval $[0,c]$ \cite{grohsliehr1, wellershoff2022injectivity}. Note, that lattices are a specific type of uniformly discrete sets, which are sets $S \subseteq \R^n$ satisfying the condition
$$
\inf_{\substack{s,s' \in S \\ s \neq s'}} |s - s'| > 0,
$$
with $|\cdot|$ the Euclidean norm in $\R^n$. Uniformly discrete sets play an essential role in sampling theory. Considering the negative result on phaseless lattice sampling in $\lt$, and the positive result in $L^2[0,c]$, it is natural to inquire whether uniqueness can be achieved from uniformly discrete sets in intermediate spaces $X$, that is,
$$
L^2[0,c] \subsetneq X \subsetneq \lt, \quad c >0.
$$
Based on Theorem \ref{thm:main}, we demonstrate that this holds true for the space of square-integrable functions supported on the positive real axis $\R_+ \coloneqq [0,\infty)$.

\begin{theorem}\label{thm:main_laplace}
    Let $\Lambda \subseteq \R$ be not contained in an arithmetic progression, and let $U=\{ u_j : j \in \N \} \subseteq \R_+$ be a sequence satisfying
    \begin{equation}
        \sum_{\substack{j=1 \\ u_j \neq 0}}^\infty \frac{1}{u_j} = \infty.
    \end{equation}
    Then $-U \times \Lambda$ is a uniqueness set for the Gabor phase retrieval problem in $L^2(\R_+)$.
\end{theorem}

In fact, a more general version of the previous result is true: it is sufficient for $U$ to be a uniqueness set for the Laplace transform. For a detailed discussion of this generalization, we refer the reader to Section \ref{sec:proof_consequences}. It is important to note that in Theorem \ref{thm:main_laplace}, both $\Lambda$ and $U$ can be chosen to be uniformly discrete. Such a choice leads to a uniformly discrete uniqueness set for the Gabor phase retrieval problem in $L^2(\R_+)$. In view of the observation that perturbations of arithmetic progressions of the form \eqref{eq:perturbation} are not contained in an arithmetic progression, we obtain as a corollary, that a perturbation of a lattice forms a uniqueness set for the intermediate space $L^2(\R_+)$.

\begin{corollary}\label{cor:perturb}
    Let $(p_n)_{n \in \Z}$ be a sequence as in \eqref{eq:perturbation}, let $\tilde \Z \coloneqq \{ n +p_n : n \in \Z \}$, and let $\alpha,\beta > 0$. Then $\alpha \Z \times \beta \tilde \Z$ is a uniqueness set for the Gabor phase retrieval problem in $L^2(\R_+)$.
\end{corollary}

The recent work \cite{perturbations} also highlighted the benefits of employing perturbed lattices over regular lattices in the uniqueness problem in Gabor phase retrieval. A pivotal difference between Corollary \ref{cor:perturb} and the perturbation results in \cite{perturbations} is that $\alpha$ and $\beta$ are arbitrary positive constants. Consequently, there is no requirement for a density assumption on the unperturbed lattice $\alpha \Z \times \beta \Z$.

We remark that the Hardy space $\mathcal H^2(\R)$ is the image of $L^2(\R_+)$ under the Fourier transform. According to \cite[Equation (3.10)]{Groechenig}, we have $|\mathcal{G}f(x,\omega)| = |\mathcal{G}\hat f(\omega,-x)|$ for every $f \in \lt$ and every $(x,\omega) \in \R^2$, whereby $\hat f$ is the Fourier transform of $f$. It follows that the results above directly transfer to $\mathcal{H}^2(\R)$. For example, $\beta \tilde \Z \times \alpha \Z$ is a uniqueness set for the Gabor phase retrieval problem in $\mathcal{H}^2(\R)$, provided that $\tilde \Z$ is a perturbation of the integers, and $\alpha,\beta > 0$.

\subsection{Wright's conjecture and Pauli-type uniqueness}

Let $B(\lt)$ denote the collection of all bounded linear operators on $\lt$. Following \cite{Jaming2023}, we say that a family $\mathcal{T} \subseteq B(\lt)$ \emph{does phase retrieval}, if for every $f,h \in \lt$ it holds that
$$
|Tf|=|Th|, \quad T \in \mathcal{T} \implies f \sim h.
$$
The classical Pauli problem in quantum mechanics examines whether $\{ I, \mathfrak F \}$ does phase retrieval. Here, $I$ represents the identity operator and $\mathfrak F$ represents the Fourier transform, defined on $L^1(\R) \cap L^2(\R)$ by
$$
\mathfrak F f (\xi) = \int_\R f(t) e^{-2\pi i \xi t} \, dt,
$$
and extends to a unitary operator on $L^2(\R)$. It is well-established, that phase retrieval does not hold for this operator set, as explicit counterexamples can be found in \cite{corbett_hurst_1977}. Wright proposed a conjecture suggesting the existence of a unitary operator $U$ such that $\{ I, \mathfrak F, U \}$ does phase retrieval \cite{VOGT1978365}. This conjecture, known as Wright's conjecture, remains an open problem to this day.  An extension of Wright's conjecture concerns the existence of three self-adjoint operators $\{ T_1, T_2, T_3 \}$ that do phase retrieval. In the recent publication \cite{Jaming2023}, Jaming and Rathmair provided an affirmative answer to the latter problem, by demonstrating that the three Fourier multiplier operators
$$
f \mapsto \mathfrak F(m_j \mathfrak F^{-1} f), \quad j \in \{ 1,2,3 \},
$$
with $m_1(t)=e^{-\pi t^2}, \, m_2(t)=2\pi t e^{-\pi t^2}, \, m_3(t)= (1-2\pi t) e^{-\pi t^2},$ do phase retrieval \cite[Theorem 1.1]{Jaming2023}. Based on the findings of the present paper, we can readily construct alternative families of self-adjoint operators doing phase retrieval. Wright's conjecture for self-adjoint operators then follows from the fact that there exists a set of three points that is not contained in an arithmetic progression. Specifically, we have the following result.

\begin{theorem}\label{thm:wright}
    Let $\Lambda \subseteq \R$, let $\varphi(t)=e^{-\pi t^2}$, and let
    $$
    T_\lambda : \lt \to \lt, \quad T_\lambda f = \mathfrak F (\varphi(\cdot - \lambda)f), \quad \lambda \in \Lambda.
    $$
    Then $\{ T_\lambda : \lambda \in \Lambda \}$ does phase retrieval, if and only if $\Lambda$ is not contained in an arithmetic progression. In particular, for $\alpha<\beta<\gamma \in \R$, the system of three self-adjoint operators 
    $
    \{ T_{\alpha} \circ \mathfrak F^{-1},T_{\beta}\circ \mathfrak F^{-1},T_\gamma\circ \mathfrak F^{-1} \}
    $
    does phase retrieval if and only if
    $
    \frac{\beta - \alpha}{\gamma - \beta} \not \in \Q.
    $
\end{theorem}

The second statement in Theorem \ref{thm:wright} depends on the property that a holomorphic function is uniquely determined by its absolute values on three parallel lines, provided that their distances satisfy an irrationality condition. In the context of entire functions of finite order, this observation is due to Jaming \cite{jamingPerscomm}.

\section{Proof of Theorem \ref{thm:main}}

Let $x,y \in \R$ such that $y \neq 0$. Further, let $n \in \Z$ be the unique integer such that $x \in [n|y|, (n+1)|y|)$. We define the real number $\{ x \}_y$ by
$$
\{ x \}_y \coloneqq x-n|y| \in [0,|y|).
$$
Note that if $y=1$, then $\{ x \}_y = \{ x \}_1$ is simply the fractional part of $x$, given by $\{ x \}_1 = x - \lfloor x \rfloor$, where $\lfloor x \rfloor \coloneqq \max \{ n \in \Z : n \leq x\}$ denotes the integral part of $x$. Recall that if $x$ is an irrational number, then Weyl's theorem states that the set
$$
\left \{ \{ nx \}_1 : n \in \Z \right \}
$$
is uniformly distributed in $[0,1)$ \cite[Example 2.1]{KuipersNiederreiter}. For an extensive exposition on uniformly distributed sequences, we refer the reader to a book by Kuipers and Niederreiter \cite{KuipersNiederreiter}. Finally, we denote by $\#A$ the number of elements in a set $A$.

\begin{lemma}\label{lma:lattice}
    Let $D \subseteq \R$ with $D \neq \{0 \}$. If there exists a real number $x \in \R$ and an element $d' \in D \setminus \{ 0 \}$ such that the set
    $$
        S \coloneqq \bigcup_{d \in D } \left \{ \{ x+dn \}_{d'} : n \in \Z \right \}
    $$
    is finite, then there exists $a \in \R$ such that $D \subseteq a \Z$.
\end{lemma}
\begin{proof}
    Let $x \in \R$ and $d' \in D \setminus \{ 0 \}$ be as in the hypothesis of the statement.
    We define for every $d \in D$ a corresponding set $S_d \subseteq \R$ via
    \begin{equation}
            S_d  \coloneqq \{ \{ x+dn \}_{d'} : n \in \Z \} = |d'| \left \{ \left \{ \frac{x}{|d'|} + \frac{d}{|d'|}n \right \}_1  : n \in \Z  \right \}.
    \end{equation}
    Let us denote the set on the right-hand side of the previous equation by $W_d$. Since $S_d \subseteq S$ and $S$ is finite, it follows that $\frac{d}{|d'|}$ is a rational number. Otherwise, $W_d$ would be uniformly distributed in $[0,1)$. Thus, $S_d$ would be an infinite set, contradicting the assumption of the statement. Consequently, we can write
\begin{equation}\label{eq:frac}
    \frac{d}{|d'|} = \ell_d + \frac{p_d}{q_d}, \quad \ell_d \in \Z, \quad p_d \in \N_0, \quad q_d \in \N,
\end{equation}
with $p_d,q_d$ co-prime and $p_d < q_d$ (for $p_d=0$ we set $q_d=1$). Using this notation, we can write the set $S_d$ as
$$
S_d  = |d'| \left \{ \left \{ \frac{x}{|d'|} + \frac{p_d}{q_d}n \right \}_1 : n \in \Z  \right \}.
$$
Since $p_d,q_d$ are co-prime, the set $S_d$ contains exactly $q_d$ elements. In combination with the property $\# S_d \leq \# S < \infty$ for every $d \in D$, it follows that the set $Q = \{ q_d : d \in D \}$ is a bounded set of integers. Hence, $Q$ is finite. Suppose that $\# Q = N$ with
$$
Q = \{ r_1, \dots, r_N \},
$$
where $r_1, \dots, r_N \in \Z$ are distinct. For $q \coloneqq \prod_{j=1}^N r_j$ and suitable $\tilde p_d \in \Z$ we have
$$
\frac{p_d}{q_d} = \frac{\tilde p_d}{q}.
$$
In view of equation \eqref{eq:frac}, it holds that
$$
\frac{d}{|d'|} = \ell_d + \frac{\tilde p_d}{q} \in \Z + \frac{1}{q}\Z \subseteq \frac{1}{q} \Z.
$$
Thus, $d \in \frac{|d'|}{q}\Z$ for every $d \in D$ which yields the desired inclusion $D \subseteq a\Z$ with $a=\frac{|d'|}{q}$.
\end{proof}

For an entire function $F : \C \to \C$ which does not vanish identically, we define the zero-counting function of $F$ by
    \begin{equation}
        m_F(z)= \begin{cases*}
      n, & if $F$ has a zero of order $n$ at $z$, \\
      0, & else.
    \end{cases*}
    \end{equation}
Further, let $Z(F)$ be the zero set of $F$, where the zeros are counted with multiplicities. Thus, if $H$ is a second entire function satisfying $Z(F)=Z(H)$, then $F,H$ have the same zeros with same multiplicities. Suppose that $\Gamma_1,\Gamma_2 \subseteq \C$ are the images of two curves in $\C$, and let the entire functions $F,H$ satisfy
\begin{equation}\label{eq:idd_gamm1_gamma2}
    |F(z)| = |H(z)|, \quad z \in \Gamma_1 \cup \Gamma_2.
\end{equation}
Previous articles have made noteworthy observations regarding the identity \eqref{eq:idd_gamm1_gamma2}, indicating periodicity properties of the zero counting functions of $F$ and $H$ \cite{perez_2021,Jaming,Wellershoff2022,Mallat2015}. For instance, it has been observed by Perez, that if $\Gamma_1$ and $\Gamma_2$ are two distinct circles centered at the origin, and if $F,H$ satisfy $\eqref{eq:idd_gamm1_gamma2}$, then the difference map $m_F - m_H$ exhibits a periodicity property \cite{perez_2021} (see also a recent contribution by Chalendar and Partington \cite{ChalendarPartington2024}). A similar feature is observed when $\Gamma_1$ and $\Gamma_2$ are parallel lines \cite{Mallat2015} or intersecting lines \cite{Jaming}. The periodicity induced by parallel lines was recently applied by Wellershoff in \cite{Wellershoff2022}, resulting in a sampled Gabor phase retrieval statement in the Paley-Wiener space. The structure of finite order entire functions with a given modulus on equidistant parallel lines was studied in \cite{lines}. 

For our purposes, we require the following statement.

\begin{lemma}\label{lma:periodicity}
    Let $X = \{ x,y \}, \, x,y \in \R, \, x \neq y$, and $A,F,H : \C \to \C$ three entire functions with $F \neq 0$ and $H \neq 0$. Then the following holds.
    \begin{enumerate}
        \item If $|e^{A(z)}| = 1$ for all $z \in \R + i X$, then $A$ is $2i(x-y)$-periodic.
        \item If $|F(z)|=|H(z)|$ for all $z \in \R + i X$, then $m_F - m_H$ is $2i(x-y)$-periodic.
    \end{enumerate}
\end{lemma}
\begin{proof}
    To prove the first claim, we observe that the condition $|e^{A(z)}| = 1$ for all $z \in \R + i X$ implies that
    $$
    \re A(z) = 0, \quad z \in \R + i X.
    $$
    Hence, the two functions
    $$
    z \mapsto iA(z+ix), \quad  z \mapsto iA(z+iy)
    $$
    are entire and real-valued on $\R$. It follows from the Schwarz reflection principle that
    $$
    iA(\overline{z}+ix) = \overline{iA(z+ix)}, \quad iA(\overline{z}+iy) = \overline{iA(z+iy)}, \quad z \in \C.
    $$
    Consequently,
    $$
    A(z) = A(z+2i(x-y)), \quad z \in \C.
    $$
    This proves the first claim. The second statement follows from \cite[Lemma 2]{Wellershoff2022}.
\end{proof}

\begin{proof}[Proof of Theorem \ref{thm:main}]
\textbf{Step 1: necessity.} We show that if $\Lambda$ is contained in an arithmetic progression, then $e^{i\theta}(\R + i \Lambda)$ is not a uniqueness set for the phase retrieval problem in $\ft$. To this end, suppose that $\Lambda \subseteq a\Z+b$. We can assume that $a \neq 0$: for $a=0$ it holds that $\Lambda = \{ b \} \subseteq b\Z$. Hence, $\Lambda \subseteq v\Z$ for some $v>0$ (if $b \neq 0$ then choose $v=b$; if $b=0$ then choose an arbitrary $v \neq 0$).

It suffices to show the existence of $F,H \in \ft$ satisfying both $F \not \sim H$ and $|F(z)| = |H(z)|$ for all $z \in e^{i\theta}(\R + i \Lambda)$. To do so, let $\mathbf{c} = (c,c') \in \C^2$, and define
\begin{equation}
    \begin{split}
        F_\mathbf{c}(z) &= c \exp \left ( \frac{\pi b}{a} i \right ) \exp \left ( -\frac{\pi e^{-i\theta}}{a} z \right ) + c' \exp \left ( -\frac{\pi b}{a} i \right ) \exp \left ( \frac{\pi e^{-i\theta}}{a} z \right ), \\
        H_\mathbf{c}(z) &= \overline{c} \exp \left ( \frac{\pi b}{a} i \right ) \exp \left ( -\frac{\pi e^{-i\theta}}{a} z \right ) + \overline{c'} \exp \left ( -\frac{\pi b}{a} i \right ) \exp \left ( \frac{\pi e^{-i\theta}}{a} z \right ).
    \end{split}
\end{equation}
Since $\ft$ is a complex vector space containing $\mathcal{E}(\C^*)$, it follows that $F_\mathbf{c},H_\mathbf{c} \in \ft$. For $n \in \Z$ and $x \in \R$, consider the complex number
\begin{equation}\label{zxn}
    z = e^{i\theta} (x+i(an+b)) \in e^{i\theta}(\R+i\Lambda).
\end{equation}
For such $z$, it holds that
\begin{equation}
    \begin{split}
        F_\mathbf{c}(z) &= c \exp \left ( -\frac{\pi x}{a} - i n \pi \right ) + c' \exp \left ( \frac{\pi x}{a} + i n \pi \right ), \\
        H_\mathbf{c}(z) &= \overline{c} \exp \left ( -\frac{\pi x}{a} - i n \pi \right ) + \overline{c'} \exp \left ( \frac{\pi x}{a} + i n \pi \right ).
    \end{split}
\end{equation}
Therefore, for $z$ of the form \eqref{zxn}, we have $|F_\mathbf{c}(z)|=|H_\mathbf{c}(z)|$, where we used that $e^{i\pi n} = e^{-i\pi n}$ for $n \in \Z$. Since $x \in \R$ and $n \in \Z$ were arbitrary, it follows that 
\begin{equation}\label{eq:id}
    |F_\mathbf{c}(z)| = |H_\mathbf{c}(z)|, \quad z \in e^{i\theta}(\R + i \Lambda).
\end{equation}
Now suppose that $F_\mathbf{c} \sim H_\mathbf{c}$ and let $\nu \in \C$, $|\nu|=1$, be such that $F_\mathbf{c} = \nu H_\mathbf{c}$. Since $F_\mathbf{c},H_\mathbf{c}$ arise as linear combinations of elements in $\mathcal{E}(\C^*)$, and since $\mathcal{E}(\C^*)$ forms a finitely linearly independent system of functions \cite[Chapter VI, Theorem 4.1]{lang2005algebra}, we obtain the relation
\begin{equation}\label{eq:nu_system}
    \begin{pmatrix} c \\ c' \end{pmatrix} = \nu \begin{pmatrix} \overline{c} \\ \overline{c'} \end{pmatrix}.
\end{equation}
It is easy to see, that a suitable choice of $(c,c')$ renders the system \eqref{eq:nu_system} unsolvable for all $\nu \in \C$ (for example pick $c=1$ and $c'=i$). Consequently, for such $(c,c')$, the assumption $F_\mathbf{c} \sim H_\mathbf{c}$ leads to a contradiction, and it must hold that $F_\mathbf{c} \not \sim H_\mathbf{c}$. At the same time, Equation \eqref{eq:id} holds true for every choice of $(c,c')$. We therefore proved the existence of two non-equivalent functions in $\ft$ with the same modulus on $e^{i\theta}(\R + i \Lambda)$.

\textbf{Step 2: sufficiency.}
Let $\Lambda$ be not contained in an arithmetic progression. In order to conclude that $e^{i\theta}(\R + i \Lambda)$ is a uniqueness set for the phase retrieval problem in $\ft$, we need to show that if $F,H \in \ft$ satisfy
\begin{equation}
    |F(z)| = |H(z)|, \quad z \in e^{i\theta}(\R + i \Lambda),
\end{equation}
then $F \sim H$. Replacing $F$ and $H$ by $F(e^{i\theta} \cdot)$ and $H(e^{i\theta} \cdot)$, respectively, we may assume without loss of generality that $\theta = 0$. Therefore, we suppose in the following, that $F$ and $H$ have the same absolute value on $\R +i\Lambda$. Finally, we assume that both $F$ and $H$ do not vanish identically. If $F$ were the zero function, then by the assumption that $F$ and $H$ have equal modulus on $\mathbb{R} + i\Lambda$, it would follow that $H$ is also the zero function, and vice versa. If $F=H=0$, then it is evident that $F \sim H$.

Note that the assumption on $\Lambda$ not being contained in an arithmetic progression, implies that $\#\Lambda \geq 3$. Pick an arbitrary $\lambda_0 \in \Lambda$ and define
$$
D \coloneqq \{ \lambda_0 - \lambda : \lambda \in \Lambda \} = \lambda_0 - \Lambda.
$$
The fact that $\Lambda$ is not contained in an arithmetic progression implies that $D$ is not contained in any set of the form $a\Z, \, a \in \R$. Further, $\#D \geq 2$ and thus there exists $d' \in D$ satisfying $d' \neq 0$. Fix such an element $d'$ and define for every $d \in D$ and $x \in \R$, a corresponding set $S_d(x)$ via
$$
S_d(x) \coloneqq \{ \{ x+2 d n \}_{2d'} : n \in \Z \}.
$$
Further, we define
$$
S(x) \coloneqq \bigcup_{d \in D} S_d(x).
$$
Observe that the set $S(x)$ satisfies $S(x) \subseteq [0,2|d'|)$. According to Lemma \ref{lma:lattice}, the set $S(x)$ is infinite. Therefore, there exists an accumulation point $\mathfrak a \in \R$ of $S(x)$.

Finally, let
\begin{equation}
        F(z) = z^m e^{f(z)} P(z), \quad
        H(z) = z^\ell e^{h(z)} Q(z),
\end{equation}
be the Weierstrass factorization of $F,H$ with $P,Q$ denoting the products of elementary factors corresponding to the zeros of $F$ and $H$, respectively. Moreover, $f,h$ are entire functions and $m,\ell \in \N_0$.

\textbf{Case 2.1: $Z(F)=Z(H)$.} Let us assume that $Z(F)=Z(H)$. In this case, we have $m = \ell$ and can assume that $P=Q$. Therefore, with $A\coloneqq f-h$, we obtain the relation
$$
|e^{A(z)}| = 1, \quad z \in \R + i\Lambda.
$$
In view of Lemma \ref{lma:periodicity}(1) and the notation above, it follows that
$$
A(z) = A(z+2id), \quad z \in \C, \quad d \in D.
$$
Thus,
$$
A(0) = A(iz), \quad z \in  S(0).
$$
The existence of an accumulation point $\mathfrak a$ of $S(0)$, in combination with the identity theorem for holomorphic functions, shows that $A$ is a constant function. Since $z \mapsto \re( A(z))$ vanishes on $\R + i\Lambda$, it follow that $A(z) = i\zeta$ for some constant $\zeta \in \R$. This yields the assertion.

\textbf{Case 2.2: $Z(F) \neq Z(H)$.} We show that $Z(F) \neq Z(H)$ contradicts the assumption that both $F$ and $H$ do not vanish identically. To do so, we observe that if $Z(F) \neq Z(H)$, then there exists a complex number $z_0 = x_0 + i y_0 \in \C$ such that $m_F(z_0) \neq m_H(z_0)$. Without loss of generality, we may assume that
\begin{equation}\label{eq:neq}
    0 < m_F(z_0) - m_H(z_0).
\end{equation}
In a similar fashion as in the previous step, we use an induced periodicity, this time of the difference map $m_F - m_H$: Lemma \ref{lma:periodicity}(2) implies that
$$
0 < m_F(z_0) - m_H(z_0) = m_F(x_0 + iz) - m_H(x_0 + iz), \quad z \in  S(y_0).
$$
Let $\mathfrak a$ be an accumulation point of $S(y_0)$, and $\{ s_k : k\in \N \} \subseteq S(y_0)$ a sequence of distinct elements such that $s_k \to \mathfrak a$. Then for every $k \in \N$ it holds that $m_F(x_0 + is_k) > m_H(x_0 + is_k) \geq 0$. In particular, every $x_0 + is_k$ is a zero of $F$. The identity theorem for holomorphic functions implies that $F$ vanishes identically. From the assumption that $F$ and $H$ have the same absolute value on $\R +i\Lambda$, it follows that also $H$ vanishes identically. This yields the desired contradiction. The proof is complete.
\end{proof}

\begin{remark}\label{rem:acc_point}
Observe that if $F$ is an entire function, then the map
$
z \mapsto |F(z)|^2 = F(z)\overline{F(z)}
$
extends from $\R$ to an entire function. Hence, for an equality of the form
$$
    |F(z)| = |H(z)|, \quad z \in \R + i \Lambda
$$
to hold, it suffices that 
$$
    |F(z)| = |H(z)|, \quad z \in \Gamma + i \Lambda,
$$
where $\Gamma$ is any set of uniqueness for entire functions. For instance, a set which contains an accumulation point does the job.
\end{remark}

\begin{remark}
It is evident, that if $\mathcal{U}$ is a uniqueness set for the phase retrieval problem in $\ft$, and if $\ft' \subseteq \ft$, then $\mathcal{U}$ is a uniqueness set for the phase retrieval problem in $\ft'$. If $\mathcal{O}$ denotes the collection of all entire functions, then $\mathcal{O}$ satisfies the assumptions of Theorem \ref{thm:main}. Consequently, if $\Lambda \subseteq \R$ is not contained in an arithmetic progression, and if $\theta \in \R$, then $e^{i\theta}(\R + i \Lambda)$ is a uniqueness set for the phase retrieval problem in $\ft'$, where $\ft'$ is an arbitrary set of entire functions.
\end{remark}

\section{Proof of the consequences of Theorem \ref{thm:main}}\label{sec:proof_consequences}

Let $\mathbb H = \{ z \in \C : \re z > 0 \}$ denote the open right half-plane, and $\overline{\mathbb H} = \{ z \in \C : \re z \geq 0 \}$ the closed right half-plane. The Laplace transform of a function $f \in L^1(\R_+)$ at $s \in \overline{\mathbb H}$ is given by
$$
\mathcal L f(s) \coloneqq \int_{\R_+}f(y)e^{-sy} \, dy.
$$
The function $\mathcal Lf$ defines a holomorphic function on $\mathbb H$. Recall that the Laplace transform enjoys the following convolution theorem: if $f,h \in L^1(\R_{+})$, and $s \in \mathbb{H}$, then
    $
    \mathcal L(f*h)(s) = \mathcal L f(s)\mathcal L h(s),
    $
 where $f * h \in L^1(\R_+)$ denotes the convolution of $f$ and $h$,
 $$
f*h(x)  =\int_0^x f(y)h(x-y) \, dy.
$$
We say that $U  \subseteq \R_+$ is a uniqueness set for the Laplace transform if a function $f \in L^1(\R_+)$ vanishes identically, provided that $\mathcal L f$ vanishes on $U$. The classical uniqueness theorem for the Laplace transform asserts that $\R_+$ is a uniqueness set for the Laplace transform. A theorem due to Lerch states that every set of the form $a\N +b$ with $a,b >0$ is a uniqueness set for the Laplace transform \cite{Lerch}. In fact, Lerch's theorem is a consequence of a more general principle related to zeros of analytic functions which are bounded on $\mathbb H$. It goes back to Pólya and Szegö \cite[Problem 298]{polya}. A proof can be found in a classical paper by Titchmarsh on zeros of integral transforms \cite[Theorem II]{Titchmarsh}.

\begin{theorem}[Pólya-Szegö]\label{thm:ps}
    Suppose that $F : \mathbb H \to \C$, $F \neq 0$, is holomorphic, and has zeros $z_1,z_2, \dots $ with $z_j = r_j e^{i \theta_j}$. If $F$ is bounded, then
    $$
    \sum_{j=1}^\infty \frac{\cos\theta_j}{r_j} < \infty.
    $$
\end{theorem}

A theorem on uniqueness for the Laplace transform can be deduced directly from Theorem \ref{thm:ps}. This is the content of the next statement.

\begin{corollary}\label{cor:laplace_polya}
    Suppose that $U = (u_j)_{j \in \N} \subseteq \R_+$ is a sequence such that
    \begin{equation}\label{eq:divergence_assumption}
        \sum_{\substack{j=1 \\ u_j \neq 0}}^\infty \frac{1}{u_j} = \infty.
    \end{equation}
    Then $U$ is a uniqueness set for the Laplace transform. 
\end{corollary}
\begin{proof}
    Let $f \in L^1(\R_+)$ such that $\mathcal{L}f(u_j)=0$ for every $j \in \N$. We need to show that $f=0$. To do so, we observe that for every $z \in \mathbb H$ one has
    $$
    |\mathcal{L}f(z)| \leq \| f \|_{L^1(\R_+)}.
    $$
    Hence, if $f$ would not be the zero function, then by the classical uniqueness theorem for the Laplace transform, $\mathcal{L}f$ would be a holomorphic function on $\mathbb H$ which does not vanish identically. However, the divergence assumption \eqref{eq:divergence_assumption} then leads to a contradiction to Theorem \ref{thm:ps}. Thus, $f=0$.
\end{proof}

The Bargmann transform of $f \in \lt$ is the entire function $\mathcal{B}f$, defined by
$$
\mathcal{B}f(z) = \int_\R f(t) e^{2\pi t z - \pi t^2 - \frac{\pi}{2}z^2} \, dt.
$$
The operator $\mathcal{B}$ is a unitary map from $\lt$ onto the Bargmann-Fock space $\mathcal{B}(\lt) \coloneqq \{ \mathcal{B} f : f \in \lt \}$. This space contains all exponentials $\mathcal{E}(\C)$. For functions supported on a half-line, we have the following elementary relation between the operators $\mathcal{G}, \, \mathcal{B}$, and $\mathcal{L}$.

\begin{lemma}\label{lma:GBLrelation}
    Let $f \in L^2(\R_+)$ and $z=x+i\omega \in \C$. Then
    \begin{equation}\label{eq:rel}
        \mathcal{G}f(x,-\omega) = e^{\pi i x \omega - \frac{\pi}{2}|z|^2} \mathcal{B}f(z) = \varphi(x) \cdot \mathcal{L}(f\varphi e^{2\pi i \omega (\cdot)})(-2\pi x).
    \end{equation}
\end{lemma}
\begin{proof}
    The first identity is well-known and holds for every square-integrable function \cite[Proposition 3.4.1]{Groechenig}. The factorization $\varphi(t-x)=\varphi(x)\varphi(t)e^{2\pi t x}$ shows that
    $$
    \mathcal{G}f(x,-\omega) = \varphi(x) \int_{\R_+} f(t) \varphi(t) e^{2 \pi i \omega t} e^{2\pi t x} \, dt = \varphi(x) \cdot \mathcal{L} (f \varphi e^{2 \pi i \omega (\cdot)})(-2\pi x),
    $$
    which yields the second identity.
\end{proof}

In the remainder of the article, we utilize the property, that the equivalence relation "$\sim$" is invariant under composition with injective linear operators on $\lt$. In other words, if $Q : \lt \to Y$ is an injective linear operator from $\lt$ into some function space $Y$, then $f \sim h$, if and only if $Qf \sim Qh$. We refer to this property as the \emph{invariance property}. The invariance property, when combined with the first equality in \eqref{eq:rel} and Theorem \ref{thm:main}, directly implies Theorem \ref{thm:gabor_impl} from Section \ref{sec:1}. We proceed with the proof of the Gabor phase retrieval result in the intermediate space $L^2(\R_+)$.

\begin{theorem}\label{thm:general_laplace}
    Let $U  \subseteq \R_+$ be a uniqueness set for the Laplace transform, and let $\Lambda \subseteq \R$ be not contained in an arithmetic progression. Then $-U \times \Lambda$ is a uniqueness set for the Gabor phase retrieval problem in $L^2(\R_+)$.
\end{theorem}
\begin{proof}
    Suppose that $f,h \in L^2(\R_+)$ satisfy
\begin{equation}\label{eq:ident}
    |\mathcal{G}f(z)| = |\mathcal{G}h(z)|, \quad z \in -U  \times \Lambda.
\end{equation}
We need to show that $f \sim h$. To that end, fix an element $\omega \in \Lambda$ and let $u \in U$. Combining equation \eqref{eq:ident} with Lemma \ref{lma:GBLrelation}, and taking squares implies that
$$
|\mathcal{L}(f\varphi e^{-2\pi i \omega (\cdot)})(u)|^2 = |\mathcal{L}(h\varphi e^{-2\pi i \omega (\cdot)})(u)|^2, \quad u \in 2\pi U.
$$
Using the convolution theorem for the Laplace transform and the fact that $f\varphi \in L^1(\R_+)$, yields the identity
$$
|\mathcal{L}(f\varphi e^{-2\pi i \omega (\cdot)})(u)|^2  =  \mathcal{L}(f\varphi e^{-2\pi i \omega (\cdot)} * \overline f \varphi e^{2\pi i \omega (\cdot)})(u).
$$
An analogous identity holds for $f$ replaced by $h$. Since a dilation of a uniqueness set for the Laplace transform remains a uniqueness set, it follows that $2\pi U $ is a uniqueness set for the Laplace transform. As a result, we have that
$$
|\mathcal{L}(f\varphi e^{-2\pi i \omega (\cdot)})(u)|^2 = |\mathcal{L}(h\varphi e^{-2\pi i \omega (\cdot)})(u)|^2, \quad u \in \R_+.
$$
We now use the second equality in Lemma \ref{lma:GBLrelation}, which relates the Bargmann transform to the Laplace transform. In view of this identity and the arbitrariness of $\omega \in \Lambda$, we obtain
$$
|\mathcal{B}f(z)|^2 = |\mathcal{B}h(z)|^2, \quad z \in (-\infty,0] - i \Lambda.
$$
Observe that $-\Lambda$ is not contained in an arithmetic progression, and that $(-\infty,0]$ contains an accumulation point. A combination of Theorem \ref{thm:main} and Remark \ref{rem:acc_point}, together with the fact that $\mathcal{B}f$ and $\mathcal{B}h$ are entire functions, shows that $\mathcal B f \sim \mathcal B h$. The invariance property implies $f \sim h$.
\end{proof}

According to Corollary \ref{cor:laplace_polya}, every set $U$ satisfying \eqref{eq:divergence_assumption} is a uniqueness set for the Laplace transform. Hence, Theorem \ref{thm:main_laplace} is a direct consequence of Theorem \ref{thm:general_laplace} and Corollary \ref{cor:laplace_polya}. We conclude with the proof of Theorem \ref{thm:wright} from Section \ref{sec:1}.

\begin{proof}[Proof of Theorem \ref{thm:wright}]
    Observe that for any $f \in \lt$ it holds that $T_\lambda f(t) = \mathcal{G}f(\lambda,t)$. In view of Lemma \ref{lma:GBLrelation}, the condition
    $$
    |T_\lambda f| = |T_\lambda h|, \quad \lambda \in \Lambda
    $$
    for some $f,h \in \lt$ is equivalent to
    $$
    |\mathcal{B}f(z)| = |\mathcal{B}h(z)|, \quad z \in \Lambda + i \R = e^{i\frac{\pi}{2}}(\R - i \Lambda).
    $$
    Hence, Theorem \ref{thm:main} shows that $\{ T_\lambda : \lambda \in \Lambda \}$ does phase retrieval if and only if $-\Lambda$ is not contained in an arithmetic progression. The latter is equivalent to $\Lambda$ not being contained in an arithmetic progression.

    In view of the invariance property, $\mathcal{S} \coloneqq \{ T_{\alpha} \circ \mathfrak F^{-1},T_{\beta}\circ \mathfrak F^{-1},T_\gamma\circ \mathfrak F^{-1} \}$ does phase retrieval if and only if $\{ T_\alpha, T_\beta, T_\gamma \}$ does phase retrieval. According to the first part of the statement, this is equivalent to the condition that the set $\{\alpha,\beta,\gamma \}$ is not contained in an arithmetic progression, which happens precisely in the situation, when the two distances $\beta - \alpha$, and $\gamma - \beta$ are linearly independent over $\Q$. Equivalently, $\frac{\beta-\alpha}{\gamma-\beta} \not \in \Q$. 
\end{proof}

\bibliographystyle{abbrv}
\bibliography{bibfile}

\end{document}